\documentclass[12pt]{amsart}
\usepackage{amsmath,amsthm,amsfonts,amssymb,latexsym,enumerate,xcolor}
\usepackage[pagebackref]{hyperref}
\usepackage[english]{babel}

\newcommand{\F}{\mathbb{F}}

\newcommand{\diag}{{{\operatorname{diag}}}}

\newcommand{\Irr}{{{\operatorname{Irr}}}}
\newcommand{\GL}{\operatorname{GL}}

\newcommand{\SL}{\operatorname{SL}}
\newcommand{\PSL}{\operatorname{PSL}}

\newcommand{\Syl}{\operatorname{Syl}}

\newtheorem{thm}{Theorem}[section]
\newtheorem{lem}[thm]{Lemma}

\newtheorem{cor}[thm]{Corollary}
\newtheorem{que}[thm]{Question}

\newtheorem*{thmA}{Theorem A}
\newtheorem*{conA'}{Conjecture A'}
\newtheorem*{thmB}{Theorem B}
\newtheorem*{thmC}{Theorem C}
\newtheorem*{thmD}{Theorem D}
\newtheorem*{thmE}{Theorem E}

\theoremstyle{definition}

\numberwithin{equation}{section}

\begin{document}

%%%%%%%%%%%%%%%%%%%%%%%%%%%%%%%%%%%%%%%%%%%%%%%%%%%%%%%%%%%%%
\title[Covering $p$-elements by Sylow $p$-subgroups]{Covering the set of $p$-elements in finite groups \\by Sylow $p$-subgroups}
%%%%%%%%%%%%%%%%%%%%%%%%%%%%%%%%%%%%%%%%%%%%%%%%%%%%%%%%%%%%%

\author{Attila Mar\'oti}
\address{Alfr\'ed R\'enyi Institute, Re\'altanoda Utca 13-15, H-1053, Budapest, Hungary}
\email{maroti@renyi.hu}

\author{Juan Mart\'{\i}nez}
\address{Departament de Matem\`atiques, Universitat de Val\`encia, 46100
  Burjassot, Val\`encia, Spain}
\email{Juan.Martinez-Madrid@uv.es}

\author{Alexander Moret\'o}
\address{Departament de Matem\`atiques, Universitat de Val\`encia, 46100
  Burjassot, Val\`encia, Spain}
\email{alexander.moreto@uv.es}

\thanks{Parts of this work were done when the second and third authors 
were visiting the Alfr\'ed R\'enyi Institute and when the first author was 
visiting the University of Val\`encia. We thank the Alfr\'ed R\'enyi Institute 
and the CARGRUPS research team at the University of Val\`encia for their 
hospitality. All authors were supported by Ministerio de Ciencia e 
Innovaci\'on (Grants PID2019-103854GB-I00 and PID2022-137612NB-I00 funded by 
MCIN/AEI/10.13039/501100011033 and ``ERDF A way of making Europe") and CIAICO/2021/163. The first author has also 
received funding from the European Research Council (ERC) under the European Union's Horizon 2020 research and 
innovation programme (grant agreement No 741420) and was also supported 
by the National Research, Development and Innovation Office (NKFIH) 
Grant No.~K138596, No.~K132951 and Grant No.~K138828. The second author 
was also supported by CIACIF/2021/228. Last, but not least, we thank the referee for the careful reading of the paper and helpful comments.}



\date{\today}

\begin{abstract}
Let $G_p$ be the set of $p$-elements of a finite group $G$. Do we need all the Sylow $p$-subgroups of $G$ to cover $G_p$? 
Although this question does not have an affirmative answer in general, 
our work indicates that the answer is yes more often than one could perhaps expect. 

\end{abstract}

\maketitle

%%%%%%%%%%%%%%%%%%%%%%%%%%%%%%%%%%%%%%%%%%%%%%%%%%%%%%%%%%%%%%%%%%%%%%%%%

\section{Introduction}

It is an elementary fact that a group cannot be expressed as the
union of two proper subgroups. Let $G$ be a noncyclic finite
group. Cohn \cite{coh} introduced the invariant $\sigma(G)$ as the minimal size
of a covering for $G$ which consists of proper subgroups of $G$.
Tomkinson \cite{tom} proved that $\sigma(G)$ is always a prime power plus $1$
for any (noncyclic and finite) solvable group $G$. There is a
large literature on $\sigma$. The numbers $\sigma(G)$ were computed (or
bounds were given) for various classes of nonsolvable groups
$G$; for certain symmetric groups \cite{mar}, \cite{Sw}, \cite{FGM}, for certain linear groups
\cite{BFS}, \cite{BEGHM}, for sporadic groups \cite{HM}, for Suzuki groups \cite{Lucido}, or
for certain wreath products \cite{GarMar, ag}. There are many positive integers $m$ for
which there is no group $G$ with $\sigma(G) = m$
(see \cite{tom}, \cite{Garonzi}, \cite{GKS}).

Let $p$ be a prime and let $G$ be a finite group. 
We write $G_p$ to denote the set of $p$-elements of $G$.
Motivated by recent work on the commuting probability of $p$-elements \cite{bgmn} (i.e., a local version of the well-known commuting probability in finite groups), 
we study the 
minimal size of a covering of $G_p$ by proper subgroups in \cite{mmm}. 
As one could expect, this local version lies even deeper than the global problem of studying $\sigma(G)$.

There is another version of the problem of covering $G_p$, for a finite
group $G$ and a prime $p$, that seems very natural and will
be considered in this paper: covering $G_p$ by Sylow $p$-subgroups. Since every
$p$-element belongs to some Sylow $p$-subgroup, we clearly have that
$G_p$ can be covered by the set $\mathrm{Syl}_{p}(G)$ of Sylow
$p$-subgroups of $G$. We say that $G$ has (or possesses) a redundant
Sylow $p$-subgroup if $G_p$ has a cover which is a proper subset of
$\mathrm{Syl}_{p}(G)$.

Calculations in GAP \cite{gap} suggest that perhaps surprisingly, groups with
a redundant Sylow $p$-subgroup are rare. Among the groups in
the SmallGroups library in \cite{gap}, there are only examples of groups with a
redundant Sylow $p$-subgroup when $p = 2$. The smallest of
them have order $108$ and are {\tt SmallGroup}(108,17) and {\tt SmallGroup}(108,40).

As we will see in Lemma \ref{equi}, a group $G$ does not have a redundant Sylow
$p$-subgroup for a prime $p$ if and only if there exists an
element of $G$ that belongs to a unique Sylow $p$-subgroup of $G$. This
is the case when, for example, a Sylow $p$-subgroup $P$ of $G$
is normal in $G$, or when it is cyclic, or when $P \cap Q = 1$ for any
Sylow $p$-subgroup $Q$ of $G$ but different from $P$. 

Another important case are groups of Lie type in characteristic $p$. 
We will see in Corollary \ref{lie} that, as pointed out to us by Thomas Weigel, to whom we thank,  it is easy to deduce from Lemma \ref{equi} that 
groups of Lie type in characteristic $p$ do not have a redundant Sylow $p$-subgroup.

The
elementary characterization Lemma \ref{equi} is fundamental for our work. From
this point of view, this had been studied in \cite{her, hv1, hv2,  sch} and we
think that it deserves further study. Using this characterization, it
follows from \cite{hv1} that symmetric groups do not have redundant Sylow
$p$-subgroups for any prime $p$. Similarly, in a Math StackExchange discussion, J. Schmidt
mentioned in \cite{sch} that he had constructed solvable groups with elementary
abelian Sylow $p$-subgroups where every $p$-element belongs to more than
one Sylow $p$-subgroup, but unfortunately, this does not seem
to have appeared in print.

Our first  result shows that for any prime $p$ there is a wealth of
$p$-groups that can occur as the Sylow $p$-subgroup of a solvable group with a
redundant Sylow $p$-subgroup.

\begin{thmA}
	Let $p$ be a prime. If $P$ is a non-cyclic finite $p$-group
	of exponent $p$, then there exists a solvable group $G$ with Sylow
	$p$-subgroups isomorphic to $P$ such that $G$ has a redundant Sylow
	$p$-subgroup. 
\end{thmA}

%We are not aware of $p$-groups of exponent larger than $p$ that occur as the Sylow $p$-subgroup %of a solvable group with a redundant Sylow $p$-subgroup, but we think that they might exist.
Our proof of  Theorem A relies on a deep construction of A. Turull \cite{tur}. 

Our second main result is a full characterization of groups $G$ isomorphic to a symmetric group or to an alternating group for which $G$ has a redundant Sylow $p$-subgroup. As noted before, it follows from \cite{hv1} that no symmetric
group $G$ has this property.

\begin{thmB}
	Let $p$ be a prime and let $G$ be $A_{n}$ or $S_{n}$ with $n
	\geq \max\{ 6, p \}$. The group $G$ has a redundant Sylow $p$-subgroup if and only if $p=2$  and $G = A_{n}$ with $n=\sum_{i=r}^{k}a_i2^{i}$, where $a_{r}, a_{r+1}, \ldots , a_{k} \in \{0,1\}$, $a_{r} = a_{k} =1$ and the following conditions are satisfied:
	
	\begin{itemize}
\item  $\sum_{i=1}^{k}a_i\equiv 1 \pmod{2}$ if $n$ is odd.

\item $r\geq 2$ is even and $\sum_{i=r}^{k}a_i\equiv 1 \pmod{2}$ if $n$ is even.
\end{itemize}
\end{thmB}

In the case $n=5$ it is easy to check that $S_5$ and $A_5$ do not possess a redundant Sylow $p$-subgroup for $p \in \{2,3,5\}$. In the case of alternating groups, we prove Theorem B as a consequence of a  characterization of the  $2$-elements of a symmetric  group that belong to a unique Sylow $2$-subgroup. 
This result seems of independent interest.

\begin{thmC}
Let $n\geq 2$ and let $x\in (S_n)_2$. Then $x$ lies in a unique Sylow $2$-subgroup of $S_n$ if and only if $x$ has at most two fixed points and all cycles of $x$  of lengths bigger than one have different lengths.
\end{thmC}

We also obtain a full characterization of the groups $\SL(2,q)$ and $\PSL(2,q)$, where $q$ is a prime power, with a redundant Sylow $p$-subgroup.

\begin{thmD}
Let $p$ be a prime and let $G$ be $\SL(2,q)$ or $\PSL(2,q)$ with $q$ a prime power. Then $G$ has a redundant Sylow $p$-subgroup if and only if $p=2$ and $q$ is none of the following: $2^{k}-1$, $2^{k}$ nor $2^{k}+1$ for $k$ an integer.
\end{thmD}

As mentioned by the referee, it is interesting to note that $q=9$ is the only prime power of the form $2^k\pm1$ that is not a prime. The remaining prime powers of the form $2^k\pm1$ are the Fermat and Mersenne primes. This follows from elementary number theory.

Theorems B and D, along with Weigel's observation,  suggest that  it is not easy to find almost quasisimple groups with a redundant Sylow $p$-subgroup for any odd prime. 
For instance, we do not know examples of simple groups with a redundant Sylow $p$-subgroup for $p>11$.
We have to go to the general linear groups of arbitrarily large rank to find them.

\begin{thmE}
Let $p$ be an odd prime and let $q$ be a prime power such that $p$ is the $p$-part of $q-1$. Then $\GL(p,q)$ has a redundant Sylow $p$-subgroup.
\end{thmE}

Note that the Sylow $p$-subgroups of the groups in Theorem E are $C_p\wr C_p$ (see Proposition 7.13 of \cite{sam}, for instance). In particular, they have exponent $p^2$, unlike our solvable examples. 
Theorem B provides examples of simple groups with a redundant Sylow $2$-subgroup of exponent $2^a$ with $a$ arbitrarily large, but we are not aware of such examples for odd primes.  
Our results may suggest that, perhaps, if $P$ is a
non-cyclic finite $p$-group then there exists a finite group
$G$ with Sylow $p$-subgroups isomorphic to $P$ such that $G$ has a
redundant Sylow $p$-subgroup.

Furthermore, in Section 7 we study the existence of  redundant Sylow subgroups in sporadic simple groups. In particular, we show that the Monster has a redundant Sylow $7$-subgroup, which is related to work in \cite{hv1, hv2}. 

In Section 2 we collect several criteria that will be useful to decide whether a given group has a redundant Sylow $p$-subgroup or not. Theorem A, on solvable groups,  is proved in Section 3.
We prove Theorems B and C, on alternating and symmetric groups, in Section 4. Then we prove Theorem E in Section 5 and Theorem D in Section 6.
 We conclude in Section 8 where we prove some results that relate the number of Sylow $p$-subgroups of a finite group $G$ with the existence of a redundant Sylow $p$-subgroup.
We also raise some questions, partially related to recent work of Gheri \cite{ghe} and Sambale-T\u{a}rn\u{a}uceanu \cite{st}, that we think deserve further investigation.

It may also be worth remarking that it seems that Schmidt's examples \cite{sch} were inspired from the theory of fusion systems. It seems reasonable to think that fusion systems could be helpful to study the condition provided by Lemma \ref{equi}. 
We have not pursued this here, however.

To conclude this Introduction, in this paragraph let $G$ be an infinite group. Recall that a $p$-subgroup of $G$ is a subgroup in which every element has $p$-power order and a Sylow $p$-subgroup of $G$ is a $p$-subgroup which is maximal for inclusion among all $p$-subgroups in $G$. A theorem of Neumann \cite{Neumann} states that if $G$ is the union of $m$ proper subgroups where $m$ is finite and as small as possible, then the intersection of these subgroups is a subgroup of finite index in $G$. This is the reason why one may assume that $G$ is finite when computing $\sigma(G)$, the minimal number of proper subgroups needed to cover $G$. It would be interesting to know if there is a local analogue of Neumann's theorem.  Perhaps one should assume that $G$ has only finitely many Sylow $p$-subgroups.

\section{General criteria for the existence of \\ a redundant Sylow $p$-subgroup}

Let $p$ be a prime. For a finite group $G$, let $\Syl_{p}(G)$ denote the set of Sylow $p$-subgroups of $G$, let $\nu_p(G) = |\Syl_{p}(G)|$, and let $O_{p}(G)$ denote the largest normal $p$-subgroup in $G$. For an element $x$ in a finite group $G$, let $x^{G}$ denote the conjugacy class of $x$ in $G$. The following fundamental result allows us to interpret the concept of redundant Sylow $p$-subgroups in a convenient way.

\begin{lem}
\label{equi}
Let $G$ be a finite group and let $p$ be a prime. Then $G$ does not have a redundant Sylow $p$-subgroup if and only if there exists $x\in G_p$ such that $x$ belongs to a unique Sylow $p$-subgroup.
\end{lem}

\begin{proof}
Let $\Syl_p(G)=\{P_1,P_2,\ldots, P_n\}$ with $n = \nu_{p}(G)$. The result is clear for $n=1$. Assume that $n > 1$. The group $G$ has a redundant Sylow $p$-subgroup if and only if $G_p=\bigcup_{i\neq j }P_i$ for some $j$. This happens if and only if $P_j\subseteq \bigcup_{i\neq j}P_i$, which is equivalent to saying that every element of $P_j$ lies in more than one Sylow $p$-subgroup.
\end{proof}

This has a number of consequences. Recall that a subgroup $H$ in a finite group $G$ is called a TI-subgroup (trivial intersection subgroup) if $H\cap H^g$ is $H$ or $\{1\}$ for every $g\in G$. A TI-subgroup which is also a Sylow $p$-subgroup is called a TI-Sylow $p$-subgroup. We have the following.

\begin{cor}\label{TI}
If $G$ is a group with TI-Sylow $p$-subgroups, then $G$ does not have a redundant Sylow $p$-subgroup. 
\end{cor}

\begin{proof}
If $G$ is a group with TI-Sylow $p$-subgroups, then every non-identity element from $G_p$ belongs to a unique Sylow $p$-subgroup. The result follows from Lemma \ref{equi}.
\end{proof}

A characterization of finite groups possessing TI-Sylow $2$-subgroups has been obtained by Suzuki \cite{suz}. For a similar but weaker characterization for odd primes, see the work of Ho \cite{Ho}. We note that cyclic Sylow $p$-subgroups of finite simple groups are TI-subgroups by \cite{Blau}.  

\begin{cor}
\label{op}
Let $G$ be a finite group. Let $P$ be a Sylow $p$-subgroup of $G$ and suppose that $P/O_p(G)$ is cyclic. 
Then $G$ does not have a redundant Sylow $p$-subgroup.
\end{cor}

\begin{proof}
Observe that $O_{p}(G)$ is contained in every Sylow $p$-subgroup of $G$. It follows by definition that $G$ has a redundant Sylow $p$-subgroup if and only if $G/O_{p}(G)$ does. We may thus assume that $O_{p}(G) = 1$. Now, if $x$ generates a Sylow $p$-subgroup $P$, then $P$ is the unique Sylow $p$-subgroup that contains $x$. The result follows from Lemma \ref{equi}.
\end{proof}

Next, we present a proof of Weigel's observation on groups of Lie type in characteristic $p$.

\begin{cor}
\label{lie}
A finite group of Lie type in characteristic $p$ does not have a redundant Sylow $p$-subgroup.
\end{cor}

\begin{proof}
Let $G$ be a connected reductive algebraic group defined over an algebraically closed field in positive characteristic $p$. Let $F$ be a Steinberg endomorphism. The finite group $G^F$ has a Sylow $p$-subgroup $U^F$ such that $N_{G^F}(U^{F}) = B^F$ where $B$ is an $F$-stable Borel subgroup of $G$. Each regular unipotent element of $G$ lies in a unique Borel subgroup. If the unipotent element is $F$-stable, the Borel subgroup will be also. Thus each regular unipotent element of $G^F$ will lie in just one Borel subgroup (and also in one Sylow $p$-subgroup) $B^F$ of $G^F$. See \cite[p. 131]{car}. Therefore, by Lemma \ref{equi}, groups of Lie type in characteristic $p$ do not have a redundant Sylow $p$-subgroup. 
\end{proof}

Note that $x\in G$ belongs to a unique Sylow $p$-subgroup of $G$ if and only if $x$ does not belong to any intersection of two different Sylow $p$-subgroups. Since defect groups are Sylow intersections, this suggests that this is related to block theory. For instance, there is the following connection with zeros of characters. 

\begin{cor}
Let $G$ be a group without a redundant Sylow $p$-subgroup. Then there exists $x\in G_p$ such that $\chi(x)=0$ for every $\chi\in\Irr(G)$ that does not belong to a $p$-block of full defect.
\end{cor}

\begin{proof}
Let $B$ be a $p$-block with defect group $D$ such that $|D|<|G|_p$. 
By Lemma \ref{equi}, there exists $x\in G_p$ that does not belong to the intersection of any two different Sylow subgroups. By Corollary 4.21 of \cite{navb}, $x$ does not belong to any conjugate of $D$. Now, Corollary 5.9 of \cite{navb}, implies that $\chi(x)=0$ for any $\chi\in\Irr(B)$, as wanted.
\end{proof}

Next, we obtain a characterization of the property ``there exists $x\in G$ that belongs to a unique Sylow $p$-subgroup of $G$".

\begin{lem}\label{necessary}
Let $p$ be a prime, let $P \in \Syl_p(G)$ and let $x\in P$. Then $$|x^G|\leq |x^G\cap P|\nu_p(G)$$ with equality if and only if $x$ lies in a unique Sylow $p$-subgroup. In particular, if $x\in P$ is a $p$-element lying in a unique Sylow $p$-subgroup, then $$|x^G|\geq \nu_p(G)|x^P|.$$
\begin{proof}
We know that $x^G=\bigcup_{Q \in \Syl_p(G)}x^G\cap Q$. Thus, we deduce that 

$$|x^G|\leq \sum_{Q \in \Syl_p(G)}|x^G\cap Q|=|x^G\cap P|\nu_p(G),$$
where the last equality holds because each intersection has the same size. Thus, the inequality is proved and equality holds if and only if each element in $x^G$ lies in a unique Sylow $p$-subgroup, or equivalently, if and only if $x$ lies in a unique Sylow $p$-subgroup. Now, the final part follows  from the fact that $x^P\subseteq x^G\cap P$.
\end{proof}
\end{lem}

The converse of the last assertion is not true.  The element $x=(1,2,3,4)(5,6,7,8)$ of the group $G=A_{8}$ is a counterexample since $|x^G|=1260=4\cdot 315=|x^P|\nu_2(G)$, but $x$ lies in more than one Sylow $2$-subgroup (we will prove this in Section 4).

We can also deduce the next  necessary condition, which was pointed out in \cite{sch}.

\begin{cor}
Let $G$ be a finite group that has no redundant Sylow $p$-subgroups. Then $|G_p|\geq\nu_p(G)$.
\begin{proof}
By Lemma \ref{equi}, let $x$ be an element of $G$ that belongs to a unique Sylow $p$-subgroup.  By Lemma  \ref{necessary}, we know that $|x^G|\geq \nu_p(G)|x^P|\geq \nu_p(G)$ and hence $|G_p|\geq |x^G|\geq \nu_p(G)$.
\end{proof}
\end{cor}

Therefore, if $G$ is a group such that  $|G_p|<\nu_p(G)$, then $G$ has a redundant Sylow $p$-subgroup. This was the basis of the construction in \cite{sch}. We remark that the converse of this result is not true since for $G={\tt SmallGroup}(108,17)$ we have that $G$ possesses a redundant Sylow $2$-subgroup, but $|G_2|=28>27=\nu_2(G)$. The condition  $|G_p|<\nu_p(G)$ seems very strong and we are not aware of any group with non elementary abelian Sylow $p$-subgroups that satisfies this condition.

The next necessary condition was communicated to us by Gabriel Navarro, to whom we thank.

\begin{lem}
\label{nav}
Let $G$ be a finite group. Suppose that $x$ belongs to a unique Sylow $p$-subgroup $P$ of $G$. Then $C_G(x)\subseteq N_G(P)$.
\end{lem}

\begin{proof}
Let $c\in C_G(x)$. Then $x^c=x\in P$, so $x\in P^{c^{-1}}$. Since $P$ is the unique Sylow $p$-subgroup that contains $x$, we deduce that $P=P^{c^{-1}}$, so $c\in N_G(P)$, as wanted.
\end{proof}

This will be useful to study redundant Sylow $p$-subgroups in sporadic groups. 
Again, the converse is not true. We can take $G=A_8$ and $x=(1,2,3,4)(5,6,7,8)$.

\section{Solvable groups}

Now, we consider solvable groups. We work toward a proof of Theorem A. The next result proves that the converse of the last assertion of  Lemma \ref{necessary} holds for $p$-nilpotent groups.

\begin{lem}
\label{newtrans}
Let $p$ be a prime. Suppose that $G=PN$ where $P\in\Syl_p(G)$, $N\trianglelefteq G$ and $(|P|,|N|)=1$. If $x \in P$, then $|x^G|\leq\nu_p(G)|x^P|$ with equality if and only if $P$ is the unique Sylow $p$-subgroup containing $x$.
\end{lem}

\begin{proof}
By Theorem 5.25 of \cite{isa}, we know that $x^G\cap P=x^P$ for every $x\in P$. For each $Q \in \Syl_p(G)$ we choose $y_Q\in x^G\cap Q$. Then we have that $y_{Q}^{G}\cap Q=y_{Q}^{Q}$ and that $|y_{Q}^{Q}|=|x^P|$. Thus 
$$x^G=\bigcup_{Q\in  \Syl_p(G)}x^G\cap Q=\bigcup_{Q\in  \Syl_p(G)}y_{Q}^{Q}$$
and hence $|x^G|\leq \sum_{Q\in \Syl_p(G)}|y_{Q}^{Q}|=\nu_p(G)|x^P|$ and the first part follows. Now, equality holds if and only if the union above is disjoint. This happens if and only if  for every $y \in x^G$ we have that $y$ lies in only one Sylow $p$-subgroup and this is if and only if $x$ lies only in one Sylow $p$-subgroup. 
\end{proof}

\begin{cor}
\label{cent}
Let $p$ be a prime. Suppose that $G=PN$ where $1<P\in\Syl_p(G)$, $N\trianglelefteq G$,  $(|P|,|N|)=1$ and $C_N(P)=1$. Then $G$ has a redundant Sylow $p$-subgroup if and only if $C_N(x)>1$ for every $x\in G_p$.
\end{cor}

\begin{proof}
By Lemma \ref{equi} we know that $G$ has a redundant Sylow $p$-subgroup if and only if for every $x\in G_p$, $x$ belongs to more than one Sylow $p$-subgroup of $G$. By Lemma \ref{newtrans}, this occurs if and only if $|x^G|<\nu_p(G)|x^P|$ for every $x\in P$. Since $N_G(P)=P\times C_N(P)=P$, we have that $\nu_p(G)=|N|$, by Sylow's theorems.  
Thus  $|x^G|<\nu_p(G)|x^P|$ for every $x\in P$ if and only if $C_G(x)\supset C_P(x)$ for every $x\in P$, which is equivalent to $C_N(x)>1$ for every $x\in P$. The result follows.
\end{proof}

Now, we can complete the proof of Theorem A. Recall that the Fitting subgroup $F(G)=F_1(G)$ of a finite group $G$ is the largest nilpotent normal subgroup of $G$. We define $F_i(G)/F_{i-1}=F(G/F_{i-1}(G))$ for $i>1$. The group $G$ is solvable if and only if there exists $n$ such that $F_n(G)=G$. The smallest such $n$ is called the Fitting height of $G$.

\begin{thm}
\label{asol}
Let $p$ be a prime and let $P$ be  a   $p$-group of order $p^n$, with  $n>1$, and exponent $p$.  Then there exists a solvable $p'$-group $N$ with Fitting height $n$ such that $P$ acts on $N$, $C_N(P)=1$ and $G=PN$ has a redundant Sylow $p$-subgroup.
\end{thm}

\begin{proof}
Let $P$ be a $p$-group of exponent $p$ of order $p^n$ with $n>1$. By Theorem B of \cite{tur}, there exists a solvable $p'$-group $N$ of Fitting  height $n$ such that $P$ acts on $N$ with $C_N (P) = 1$. We claim that the group $G=N \rtimes P$ has a redundant Sylow $p$-subgroup. By Corollary \ref{cent}, this is true if and only if $C_N(x)>1$ for every $1\neq x\in G_p$. 
 By the solvable case of Thompson's theorem (which was known to Higman in 1957, see Theorem 6.22 of \cite{isa}), we cannot have $C_N(x)=1$, because $x$ has prime order and $N$ is not nilpotent. The claim follows.
 \end{proof}

\section{Symmetric and alternating groups}

In this section, we prove Theorems B and C. To prove these results we need to introduce an alternative way to see the Sylow subgroups of $S_n$ (for more information see \cite{f}). Let $n=\sum_{i=0}^{k}a_{i}p^{i}$ be the $p$-adic expansion of $n$. We make $a_k$ disjoint subsets of size $p^k$ in $\{1,\ldots,n\}$. Inside each subset of size $p^k$ we make $p$ disjoint subsets of size $p^{k-1}$. We repeat this process till we get sets of size $1$. With the numbers not lying in the sets of size $p^k$ we make $a_{k-1}$ disjoint subsets of size $p^{k-1}$ and inside each of them we make subsets of size $p^i$ for $i\leq k-2$ as before. We repeat this process for each $j=k,\ldots, 0$. Taking all subsets obtained by the previous process we obtain a block structure for $n$. Now, given a block structure $B$ and $\sigma\in S_n$ we will say that $\sigma$ preserves the structure $B$ if $\sigma(b)\in B$ for every $b \in B$ and we will write that $\sigma(B)=B$. We have the following result, which follows from the results in \cite{f}.

\begin{thm}\label{SylowBlock}
	Let $n$ be an integer and let $B$ be a block structure of $n$. If we set $P=\{\sigma \in (S_n)_p|\sigma(B)=B\}$, then $P\in \Syl_p(S_n)$. In addition, each Sylow $p$-subgroup of $S_n$ can be associated with a unique block structure.
\end{thm}

We begin by proving that if $p$ is a prime, then $S_n$ does not possess a redundant Sylow $p$-subgroup. As a consequence, we will have that $A_n$ does not possess any redundant Sylow $p$-subgroup for $p$ odd. This proof can be found in \cite[Theorem 5.1]{hv1}, but we include it here for completeness.
Given $x\in S_{n}$ we know that $x$ can be expressed as the product of disjoint cycles and hence we can associate a partition of $n$ to $x$, which will be called the type of $x$.

\begin{thm}\label{PSym}
	Let $p$ be a prime and let $n\geq 2$. Then there exists $x \in (S_n)_p$ lying  in a unique Sylow $p$-subgroup of $S_n$.
	\begin{proof}
		Let $n=\sum_{i=0}^{k}a_{i}p^{i}$ be the $p$-adic expansion of $n$. We can write $\sum_{i=0}^{k}a_{i}p^{i}$ as $\sum_{i=1}^{t}p^{n_i}$, where $0\leq n_1\leq n_2\leq \ldots\leq n_t=k$. We choose $x\in S_{n}$ an element whose cycle structure is $(p^{n_t}, p^{n_{t-1}},\ldots,p^{n_1})$. We use induction on $n$ to prove that this permutation can only preserve one block structure of $n$.
		
		Assume first that $n=p^k$. Let $y=x^p$. In this case, $x$ is a cycle of length $p^k$ and $y$ is a product of $p$ disjoint cycles of length $p^{k-1}$. Let $B$ and $\tilde{B}$ be block structures of $n$ preserved by $x$. We know that $B$ possesses $p$ different blocks, say $b_1,\ldots, b_{p}$,  of size $p^{k-1}$  such that $x$ permutes the $b_i$ cyclically, $y(b_i)=b_i$ and $b_i=\{a,y(a),\ldots, y^{p^{k-1}-1}(a)\}$ for all $a\in b_i$. Analogously, $\tilde{B}$ possesses blocks $\tilde{b}_1,\ldots, \tilde{b}_{p}$ with the same properties. For every $j \in \{1,\ldots,p\}$ we have that there exists some $i_j$ such that $b_{i_j} \cap \tilde{b}_j\not=\varnothing $. Thus, there exists $a \in b_{i_j} \cap \tilde{b}_j$ and hence  $b_{i_j}=\{a,y(a),\ldots, y^{p^{k-1}-1}(a)\}=\tilde{b}_j$. Now, we have that the blocks of $B$ and $\tilde{B}$ lying under $\tilde{b}_j$(=$b_{i_j}$) are two block structures of $p^{k-1}$ preserved by a cycle of length $p^{k-1}$ of  $y$. Therefore, the inductive hypothesis implies that the blocks of $B$ and $\tilde{B}$ lying under $b_{i_j}$ and $\tilde{b}_j$ must coincide. Thus $B=\tilde{B}$.

		Assume now that $n \not=p^k$. Let $B$ and $\tilde{B}$ be block structures of $n$ preserved by $x$. We know that $B$ possesses $a_k$ different blocks, say $b_1,\ldots, b_{a_k}$,  of size $p^k$  such that for all $a\in b_{i}$ we have that $b_i=\{a,x(a),\ldots, x^{p^k-1}(a)\}$. Analogously, $\tilde{B}$ possesses blocks $\tilde{b}_1,\ldots, \tilde{b}_{a_k}$ with the same properties. Now, we know that 
		
		$$|b_1|+\ldots +|b_{a_k}|+|\tilde{b}_1|+\ldots +|\tilde{b}_{a_k}|=2a_kp^k\geq p^k(a_k+1)>n.$$
		Therefore, there exists $a\in b_i\cap \tilde{b}_j$ for some $i$ and $j$ and hence $b_i=\{a,x(a),\ldots, x^{p^k}(a)\}=\tilde{b}_j$. There is no loss to assume that $b_1=\tilde{b}_1=\{1,\ldots,p^k\}$.  Now, we have that $x=cy$, where $c$ is a cycle of length $p^k$ on  $\{1,\ldots, p^k\}$ and $y$ is a permutation whose cycle structure is $( p^{n_{t-1}},\ldots,p^{n_1})$ and is disjoint to $\{1,\ldots, p^k \}$. Now, since $p^k<n$, we may apply the inductive hypothesis to $c$ to deduce that $c$ can only fix a cycle structure on $p^k$, which implies that the blocks of $B$ and $\tilde{B}$ contained in $\{1,\ldots, p^k\}$ must coincide. Analogously,  applying the inductive hypothesis to $y$, we have that the blocks of $B$ and $\tilde{B}$ not contained in $\{1,\ldots, p^k\}$ must coincide. Thus $B=\tilde{B}$ and the result follows.

		Thus, $x$ can preserve only one block structure of $n$ and hence,  by Theorem \ref{SylowBlock}, we have that $x$ lies in only one Sylow $p$-subgroup of $S_{n}$.
	\end{proof}
\end{thm}

Now, it only remains to study the existence of a redundant Sylow $2$-subgroup in $A_n$. We begin by reducing the problem to $S_n$.

\begin{thm}
	Let $n\geq 6$, let $P\in \Syl_2(S_{n})$ and let $T\in \Syl_2(A_n)$. Then $N_{S_{n}}(P)=P$ and $N_{A_n}(T)=T$. In particular, $\nu_2(S_{n})=\nu_2(A_n)$ and the map $\varphi:\Syl_2(S_{n})\rightarrow\Syl_2(A_n)$  defined by $\varphi(P)=P\cap A_n$ is a bijection.
	\begin{proof}
		The assertion on the normalizers can be found in \cite[Lemma 4]{cf} and \cite[Theorem 2]{kon}. It follows trivially that $\nu_2(S_{n})=\nu_2(A_n)$. Now, we know that the map $\varphi$ is surjective and since $|\Syl_2(S_{n})|=|\Syl_2(A_n)|$, we have that it is a bijection.
	\end{proof}
\end{thm}

As a consequence, an element $x \in (A_n)_2$ lies in a unique Sylow $2$-subgroup of $A_n$ if and only if $x$  lies in a unique Sylow $2$-subgroup of $S_n$. Thus, once Theorem C is proved we will deduce Theorem B simply studying whether  the permutations of Theorem C are even or odd.

Let $n=\sum_{i=1}^{t}2^{n_i}$ with $0\leq n_1<n_2 < \ldots <n_t$ be the $2$-adic expansion of $n$. We define $x(n)$ as the partition of $n$ given by  $(2^{n_t},\ldots,2^{n_1})$. We say that $x(n)$ is even if $t$ is even and $n_{1} \geq 1$ or if $t$ is odd and $n_{1} = 0$. Otherwise $x(n)$ is called odd. Note that $x(n)$ is the partition that we considered in the proof of Theorem \ref{PSym} for $p=2$.   If $n$ is even and $n-2=\sum_{i=1}^{\ell}2^{m_i}$  with $0< m_1<m_2 < \ldots <m_{\ell}$, then  we define $y(n)$ as the partition given by $(2^{m_{\ell}},\ldots,2^{m_1},1,1)$. We say that $y(n)$ is even if $\ell$ is even and we say $y(n)$ is odd if $\ell$ is odd. We define $T(n)$ as the set of partitions  $\{x(n),y(n)\}$ if $n$ is even and $\{x(n)\}$ if $n$ is odd.

We observe that the type of $x\in (S_{n})_2$ lies in $T(n)$ if and only if $x$ fixes at most $2$ points and all  cycles in the decomposition of $x$ have different sizes. The following lemma determines whether the permutations in $T(n)$ are even or odd and hence it will provide the conditions of Theorem B.

\begin{lem}\label{num}
	Let $n$ be an integer and let $n=\sum_{i=r}^{k}a_i2^i$, where $a_r,a_{r+1},\ldots a_k \in \{0,1\}$, $1=a_r=a_k$. 
	\begin{itemize}
		\item[i)] Assume that $n$ is odd. Then an element of type $x(n)$ is an odd permutation if and only if $\sum_{i=1}^{k}a_i\equiv 1 \pmod{2}$.

		\item[ii)] Assume that $n$ is even. Then two elements of types respectively $x(n)$ and $y(n)$ are both odd permutations if and only if $\sum_{i=r}^{k}a_i\equiv 1 \pmod{2}$ and $r\geq 2$ is even.
	\end{itemize}
	
	\begin{proof}
		i) Let $n$ be odd. In this case $r=0$, $a_{0} = 1$, and $n_{1} = 0$. The partition $x(n)$ is odd if and only if $t$ is even. This is equivalent to saying that $\sum_{i=r}^{k}a_i$ is even, that is, $\sum_{i=1}^{k}a_i\equiv 1 \pmod{2}$.  
		
		ii) Let $n$ be even. In this case $n-2$ is also even and $n_{1} \geq 1$. The partitions $x(n)$ and $y(n)$ are both odd if and only if both $t$ and $\ell$ are odd. Since $n-2=\sum_{i=2}^{t}2^{n_i}+\sum_{j=1}^{r-1}2^{j}$ we deduce that $\ell=(t-1)+(r-1)=t+r-2$. Therefore, we deduce that $t$ and $\ell$ are both odd if and only if $t=\sum_{i=r}^{k}a_i$ is odd and $r$ is even.		  
	\end{proof}
\end{lem}

Now, we restate and  prove Theorem C.

\begin{thm}
	Let $n\geq 2$ and let $x\in (S_n)_2$ written as a product of disjoint cycles. Then $x$ lies in a unique Sylow $2$-subgroup of $S_n$ if and only if $x$ has at most two fixed points and all cycles of $x$ of length bigger than one have different lengths.
	\begin{proof}
		Let $x$ be a $2$-element of $S_n$ lying in a unique Sylow $2$-subgroup of $S_n$.
		
		Assume first that $x$ fixes at least $3$ points. Suppose that $n$ is odd and  that $x$ fixes the points $1,2$ and $3$ and that $x$ preserves a block structure of $n$, say $B$, such that $1$ does not lie in any block of size larger than $1$ and $\{2,3\}\in B$. Thus, if we take  another block structure, $\overline{B}$, which just permutes (non-trivially) $1,2$ and $3$, then $x$ preserves $B$ and $\overline{B}$ and hence, by Theorem \ref{SylowBlock}, $x$ lies in more than one Sylow $2$-subgroup of $S_n$, which is impossible. Assume now that $n$ is even and that $x$ fixes $4$ points. Therefore,  we can repeat the process by permuting four labels of a block structure preserved by $x$. Thus, $x$ permutes at least $2$ block structures on $n$, which is again a contradiction.
		
		Now, suppose  that $x$ possesses $2$ cycles of the same length and let $k\geq 1$ be the largest integer such that $x$ contains $2$ cycles of length $2^k$. Then $x$ preserves a block structure $B$ with $B_0,B_1,B_2,B_{11},B_{12},B_{21},B_{22}\in B$ such that $$B_0=B_1\cup B_2, B_i=B_{i1}\cup B_{i2}, x(B_{i1})=B_{i2}$$ and 
		$$2^{k-1}=|B_{ij}|=\frac{|B_i|}{2}=\frac{|B_0|}{4}.$$
		for $i,j=1,2$.
		 Thus, we can make another block structure $\overline{B}$ simply replacing the blocks inside $B_0$ by $$\overline{B_{1}}=B_{11}\cup B_{21}\text{ and\,\,} \overline{B_{2}}=B_{12}\cup B_{22}.$$ Therefore, $x$ preserves two block structures  on $n$ and hence, by Theorem \ref{SylowBlock},  lies in more than one Sylow $2$-subgroup, which is impossible.

		Now, let $x$ be a permutation whose type lies in $T(n)$.

		If the type of $x$ is $x(n)$, then  the argument of the proof of Theorem \ref{PSym} shows that $x$ lies in a unique Sylow $2$-subgroup. Assume now that $n$ is even, $$n-2=\sum_{i=1}^{\ell}2^{m_i}$$  with $0< m_1<m_2\ldots <m_{\ell}$ and the type of $x$ is $y(n)=(2^{m_{\ell}},\ldots, 2^{m_1},1,1)$ (note that  $2^{m_{\ell}}$ is the largest power of $2$ smaller than $n$). We use  induction on $n$ to prove  that $y(n)$ preserves a unique block structure of $n$. We will distinguish two cases: the case $n<2^{m_{\ell}+1}$ and $n=2^{m_{\ell}+1}$. 
		
		Assume first that $n<2^{m_{\ell}+1}$. In this case, each block structure of $n$ possesses a unique block of size $2^{m_{\ell}}$. Thus, if $B$ and $\tilde{B}$ are block structures preserved by $x$, then there exists $b\in B$ and $\tilde{b}\in B$ such that $|b|=2^{m_{\ell}}=|\tilde{b}|$. It follows that $x(b)=b$ and hence  $\{a,x(a),\ldots , x^{2^{m_{\ell}}-1}(a)\}=b$ for all $a\in b$, and a similar property holds for $\tilde{b}$. Now, we have that $n-2^{m_{\ell}}<2^{m_{\ell}}$ and hence there exists $a \in b\cap \tilde{b}$. Thus, we have that $b=\{a,x(a),\ldots , x^{2^{m_{\ell}}-1}(a)\}=\tilde{b}$.  As in the proof of Theorem \ref{PSym}, we may assume that $b=\tilde{b}=\{1,\ldots,2^{m_{\ell}}\}$. Now, we have that $x=cy$, where $c$ is a cycle of length $2^{m_{\ell}}$ on $\{1,\ldots, 2^{m_{\ell}}\}$ and $y$ is a permutation whose cycle structure is $y(n-2^{m_{\ell}})=( 2^{m_{\ell-1}},\ldots,2^{m_1},1,1)$ and is disjoint to $\{1,\ldots, 2^{m_{\ell}} \}$. Thus, applying the inductive hypothesis to $y$, we have that the blocks of $B$ and $\tilde{B}$ not contained in $\{1,\ldots,2^{m_{\ell}}\}$ must coincide. Moreover, the  blocks of $B$ and $\tilde{B}$ contained in $\{1,\ldots, 2^{m_{\ell}}\}$ must coincide since the type of $c$ is $x(2^{m_{\ell}})$.  Thus, the result follows in this case.

		Assume now that $n=2^{m_{\ell}+1}$. Let $B$ be a block structure on $n$. Then there exist $b_0,b_1,b_2\in B$ such that $|b_1|=|b_2|=2^{m_{\ell}}$,  $b_0=b_1\cup b_2$ and $b_0\neq b_1$. If $x$ preserves the structure $B$, then $x(b_1)=b_1$, $x(b_2)=b_2$ and $b_1=\{a,x(a),\ldots , x^{2^{m_{\ell}}-1}(a)\}$  for all $a\in b_1$. If $\tilde{B}$ is another block structure preserved by $x$, then there exist $\tilde{b}_0,\tilde{b}_1,\tilde{b}_2\in \tilde{B}$ satisfying the same conditions as above. 
		
		If $b_1 \cap \tilde{b}_1=\varnothing$, then $b_1=\tilde{b}_2$. Thus, if $a\in b_1$ we have that $\{a,x(a),\ldots , x^{2^{m_{\ell}}-1}(a)\}=b_1=\tilde{b}_2$ and hence both $\tilde{b}_1$ and $\tilde{b}_2$ contain a cycle of length $2^{m_{\ell}}$ of $x$, which is impossible. Hence, $b_1 \cap \tilde{b}_1\not =\varnothing$ and we can finish as in the case $n<2^{m_{\ell}+1}$.
		
		Thus, $x$ preserves a unique block structure of $n$ and therefore, by Theorem \ref{SylowBlock}, we have that $x$ lies in a unique Sylow $2$-subgroup of $S_n$.
	\end{proof}
\end{thm}

Theorem B now follows from Theorem C and Lemma \ref{num}.

\section{General linear groups}

In this section, we prove a more precise version of Theorem E. 

\begin{thm}\label{GL}
Let $p$ be an odd prime and let $q$ be a prime power such that $p$ divides $q-1$ but $p^2$ does not divide $q-1$. If $1<n<p$, then $\GL(n,q)$ does not possess a redundant Sylow $p$-subgroup and $\GL(p,q)$ possesses a redundant Sylow $p$-subgroup.
\end{thm}

We begin with a few lemmas. The first one follows from Proposition 7.13 of \cite{sam}.

\begin{lem}
If $p$ is an odd prime and $q$ is a prime power such that the $p$-part of $q-1$ is $p$, then the Sylow $p$-subgroups of $\GL(p,q)$ are isomorphic to $C_{p}\wr C_p$ and the Sylow $p$-subgroups of $\GL(n,q)$ are isomorphic to $(C_{p})^{n}$ for every integer $n$ with $1 < n < p$.
\end{lem}

We need a property of the groups $C_{p} \wr C_{p}$.

\begin{lem}
\label{char}
For every odd prime $p$, the base group $B$ of $C_{p} \wr C_{p}$ is the unique abelian maximal subgroup of $C_{p} \wr C_{p}$ and, in particular, $B$ is characteristic in $C_{p} \wr C_{p}$.
\end{lem}

\begin{proof}
Assume that $A$ is another abelian maximal subgroup in $G = C_{p} \wr C_{p}$. We have $G = AB$ and $A \cap B \leq Z(G)$ which is thus of order at most $p$. It follows that $p^{p+1} = |G| = |AB| = |A||B|/|A \cap B| \geq p^{2p-1}$, which is a contradiction.  
\end{proof}

\begin{lem}\label{diag}
Let $p$ be a prime, let $n\leq p$,  let $q$ be a prime power such that $p$ divides $q-1$ and let $$R=\{\diag(x_1,\ldots, x_n)|x_1,\ldots, x_n\in (\F_{q}^{\times})_p\}\leq \GL(n,q).$$ If $A\in \GL(n,q)$ normalizes $R$, then there exists a unique non-zero entry in each row and column of $A$.
\begin{proof}
Since $A$ normalizes $R$, we have that for each $x_1,\ldots , x_n \in  (\F_{q}^{\times})_p$ there exist $y_1,\ldots , y_n \in  (\F_{q}^{\times})_p$
 such that 
 $$\diag(x_1,\ldots , x_n )\cdot A=A\cdot \diag(y_1,\ldots , y_n).$$

  Let us take $x_1,\ldots , x_n $ such that all $x_i$ are different. Looking at the first row of the matrix equality, we have that $a_{1j}x_1=a_{1j}y_j$ for all $j$ with $1 \leq j \leq n$. There exists $j_1$ with $1 \leq j_{1} \leq n$ such that $a_{1j_{1}}\not=0$ and hence $x_1=y_{j_1}$. Now, looking at column $j_1$, we have that $a_{ij_{1}}x_i=a_{ij_{1}}y_{j_1}=a_{ij_{1}}x_1$ for all $i$ with $1 \leq i \leq n$. Since all $x_i$ are different, we have that $a_{ij_{1}}=0$ for all $i\not=1$. 
 
 Now, looking at the second row, we can repeat the argument to find $j_2\in \{1,\ldots ,n\}\setminus \{j_1\}$ such that $a_{2j_2}\not=0$ and $a_{ij_2}=0$ for all $i\not=2$. Thus, repeating the process $n$ times, we can rearrange the numbers $\{1,\ldots, n\}$ as $\{j_1,\ldots,j_n\}$ such that the only non-zero entry in the column $j_i$ is $a_{ij_i}$. Thus, $A$ has the desired form.%we find $j_1,\ldots,j_p$ different numbers in $\{1,\ldots, p\}$ such that the only non-zero entry in the column $j_i$ is $a_{ij_i}$. 
\end{proof}
\end{lem}

\begin{lem}\label{LastOne}
Let $n\leq p$ and let $q$ be a prime power such that $p$ divides $q-1$ and let $P\in \Syl_p(\GL(n,q))$. Then there exists a unique set of $1$-dimensional subspaces $\{V_1,\ldots , V_n\}$ such that $\F_{q}^{n}=V_1\oplus\ldots \oplus V_n$  and $x(V_i)\in \{V_1,\ldots , V_n\}$ for all $1\leq i \leq n$ and all $x\in P$.
\begin{proof}
We define  $$R_n=\{\diag(x_1,\ldots, x_n)|x_1,\ldots, x_n\in (\F_{q}^{\times})_p\}\leq \GL(n,q)$$ and we  also define $P_n=R_n$ if $n<p$ and $P_p=R_p\langle B\rangle$, where $B = (b_{i,j}) \in \GL(p,q)$ is the matrix defined by $b_{i+1,i}=1$ for $i<n$, $b_{1,n}=1$ and $0$ in the rest of the entries. Thus, $P_n\in \Syl_p(\GL(n,q))$  and $R_n$ is characteristic in $P_n$ for all $n\leq p$ by Lemma \ref{char}. Thus, $P$ possesses a characteristic subgroup $R$  which is conjugate to $R_n$. In particular, $R$ determines a unique set of $1$-dimensional spaces $\{V_1,\ldots , V_n\}$ such that $\F_{q}^{n}=V_1\oplus\ldots \oplus V_n$ and $x(V_i)=V_i$ for all $x\in R$ and all $i$ with $1 \leq i\leq n$.  Thus, $x(V_i)\in \{V_1,\ldots, V_n\}$ for all $x\in P$ and all $i$ with $1 \leq i\leq n$. It only remains to prove that it is the unique set of $1$-dimensional spaces satisfying these properties.

Suppose that there exist two sets $\{V_1,\ldots, V_n\}$ and $\{W_1,\ldots, W_n\}$ consisting of $1$-dimensional spaces such that $x(V_i)\in \{V_1,\ldots, V_n\}$ and $x(W_i)\in\{W_1,\ldots, W_n\}$ for all $i$ with $1 \leq i \leq n$ and all $x\in P$. Let $v_{1}, \ldots , v_{n}$ and $w_{1}, \ldots , w_{n}$ be vectors such that $V_i=\langle v_i \rangle$ and $W_i=\langle w_i \rangle$ for all $i$ with $1 \leq i \leq n$ and let $A$ be the matrix of change of base from the base $\{w_1,\ldots, w_n\}$ to the base $\{v_1,\ldots, v_n\}$. It follows that $A$ normalizes $P$ and since $R$ is characteristic in $P$ by Lemma \ref{char} we deduce that $A$ normalizes $R$. Therefore, by Lemma \ref{diag}, each row and column of $A$ possesses exactly one non-zero entry and hence $\{V_1,\ldots, V_n\}=\{W_1,\ldots, W_n\}$.
\end{proof}
\end{lem}

\begin{proof}[Proof of Theorem \ref{GL}]
By  Lemmas \ref{LastOne} and \ref{equi}, we only have to study whether there exists a $p$-element $x$ such that $x$ admits a unique decomposition $\F_{q}^{n}=V_1\oplus\ldots \oplus V_n$ with each $V_{i}$ of dimension $1$ and $x(V_{i}) \in \{ V_{1}, \ldots , V_{n} \}$ for every $i$ with $1 \leq i \leq n$.

Assume first that $n<p$. Let $P\in \Syl_p(\GL(n,q))$ and let $\{V_1,\ldots, V_n\}$ be the set of $1$-dimensional subspaces associated to $P$. Any $x \in P$ permutes the subspaces in $\{V_1,\ldots, V_n\}$, but since $n<p$, the only possibility is that $x(V_i)=V_i$ for every $i$ with $1 \leq i \leq n$. Let us take $x\in P$ such that the $V_i$ are eigenspaces of $x$ associated to different eigenvalues. If $x$ admits another decomposition with $\{W_1,\ldots, W_n\}$, then the $W_i$ must again be eigenspaces of $x$ and since the $V_{i}$ are all different eigenspaces, we have that $\{V_1,\ldots, V_n\}=\{W_1,\ldots, W_n\}$. Therefore, this $x$ lies in a unique Sylow $p$-subgroup of $\GL(n,q)$ by Lemma \ref{equi}.

Assume now that $n=p$. Let $P\in \Syl_p(\GL(n,q))$ and let $\{V_1,\ldots, V_n\}$ be the set of $1$-dimensional subspaces of $\F_{q}^{n}$ associated to $P$. Let $v_{1}, \ldots , v_{n}$ be vectors with $V_{i} = \langle v_{i} \rangle$ for every $i$ with $1 \leq i \leq n$. The group $P$ acts on $\{V_1,\ldots, V_n\}$. Let $x \in P$. There are two possibilities for the action of $x$ on $\{V_1,\ldots, V_n\}$.

\begin{itemize}
\item [a)] Each $V_i$ is an eigenspace of $x$ associated to an eigenvalue $\epsilon_i$, for some $\epsilon_i \in  (\F_{q}^{\times})_p$.

\item [b)] $x$ permutes the subspaces $V_i$ cyclically.
\end{itemize} 

Assume that we are in case a). Assume first that we have an eigenvalue of $x$ with multiplicity at least $2$. In this case we may assume that $\epsilon_1=\epsilon_2$ and take $w_1=v_1+v_2$ and $w_i=v_i$ for every $i>1$. For each $i$ with $1 \leq i \leq n$, let $W_i=\langle w_i \rangle$. We have that  $\F_{q}^{n}=W_1\oplus\ldots \oplus W_n$ and $x(W_i)=W_i$ for all $i\in \{1,\ldots, n\}$. Apply Lemma \ref{equi}.

Assume now that all eigenvalues $\epsilon_i$ are different. Since $p^2$ does not divide $q-1$, we have that $\epsilon_i^{p}=1$ for all $i$ (in fact $\{\epsilon_1,\ldots,\epsilon_p\}=(\F_{q})_p$).  Now, we take $w_i=(\epsilon_1)^{i-1}v_1+\ldots+(\epsilon_p^{i-1})^{p-1}v_p$ and $W_i=\langle w_i \rangle$ for $i\in \{1,\ldots, p\}$. Therefore, the coordinates of $v_i$ in the base $w_i$ is the transpose of the Vandermonde matrix with powers of the $\epsilon_i$. As a consequence, $\{w_1,\ldots,w_p\}$  is a base of $\F_{q}^{p}$. Therefore, we have that $\F_{q}^{p}=W_1\oplus\ldots \oplus W_p$ and, by definition, we also have that  $x(W_i)=W_{i+1}$ if $i<p$ and $x(W_p)=W_{1}$. Apply Lemma \ref{equi}.

Assume now that we are in case b). We may assume that $x(v_i)=v_{i+1}$ for $i < p$ and $x(v_{p}) = av_{1}$ for some $a\neq 0$. Suppose first that $a=1$. Let the matrix of $x$ in the base $\{v_1,\ldots,v_p\}$ be $B$. Since the characteristic polynomial of $B$ is $x^p-1$, we deduce that the set of eigenvalues of $B$ is $\{1,\epsilon,\ldots, \epsilon^{p-1}\}$, where $\epsilon \in \F_{q}^{p}$ is an element of order $p$. In particular, we deduce that  $B$ may be diagonalized  and that for all $i \in \{0,1,\ldots,p-1\}$ the dimension of the eigenspace associated to $\epsilon^i$ is $1$. Now, for $i \in \{1,\ldots,p\}$, we take an eigenvector $w_i$ associated to $\epsilon^{i-1}$ and set $W_i=\langle w_i \rangle$. We have $\F_{q}^{p}=W_1\oplus\ldots \oplus W_p$ and $x(W_i)=W_i$ for all $i$ with $1 \leq i \leq p$. Apply Lemma \ref{equi}.

Now, assume that $a\neq1$. Take
\begin{align*}
w_1 &= v_1 + v_2 + \ldots + v_p; \\
w_2 &= x(w_1) = a v_1 + v_2 + \ldots + v_p; \\
&\;\, \vdots \\
w_p &= x^{p-1}(w_1) = a v_1 + av_2 + \ldots + a v_{p-1} + v_p.
\end{align*}
Note that
\[
x(w_p) = a w_1, \quad x(w_i) = w_{i+1} \text{ for } i < p,
\]
and \(\{w_1, \ldots, w_p\}\) is a basis of \(V\) (because $a\neq1$). Therefore, \(x\) stabilizes two sets of 1-dimensional subspaces, namely $\{V_1,\dots,V_p\}$ and $\{\langle w_1\rangle,\dots,\langle w_p\rangle\}$.

We conclude that every $p$-element $x$ in $\GL(p,q)$ lies in more than one Sylow $p$-subgroup of $\GL(p,q)$.
\end{proof}

\section{The groups $\SL(2,q)$ and $\PSL(2,q)$}

Our goal in this section  is to study the existence of a redundant Sylow $p$-subgroup in the groups $\SL(2,q)$ and $\PSL(2,q)$. That is, our goal is to prove Theorem D, which we restate here.

\begin{thm}\label{PSL}
Let $p$ be a prime and let $G$ be  $\SL(2,q)$  or  $\PSL(2,q)$ with $q$  a prime power. The group $G$ possesses a redundant Sylow $p$-subgroup if and only if $p=2$ and $q$ is  none of the following.

\begin{itemize}
\item[a)]$q=2^k$ for an integer $k$.

\item[b)] $q=2^k+1$ for an integer $k$.

\item[c)]  $q=2^k-1$ for an integer $k$.
\end{itemize}
\end{thm}

If $p$ divides $q$, then the Sylow $p$-subgroups of $\SL(2,q)$ and of $\PSL(2,q)$ intersect trivially and hence, applying Corollary \ref{TI}, we have that none of $\SL(2,q)$, $\PSL(2,q)$ can possess a redundant Sylow $p$-subgroup. Thus, in the remaining, we will assume that $p$ does not divide $q$.  Moreover, if $p$ is odd and does not divide $q$, then by  Dickson's classification of subgroups of $\SL(2,q)$ and $\PSL(2,q)$ (see  \cite[Chapter XII]{Dickson}), we have that both $\SL(2,q)$ and $\PSL(2,q)$ possess cyclic Sylow $p$-subgroups and hence, by Corollary \ref{op}, none of them can have a redundant Sylow $p$-subgroup.

It only remains to consider the case when $p=2$ and $q$ is odd. In this case, we know that $|Z(\SL(2,q))|=2$ and that $Z(\SL(2,q))=O_2(\SL(2,q))$. Let $\overline{\cdot}$ denote the quotient by $Z(\SL(2,q))$. The following result relates the Sylow $2$-subgroups of $\SL(2,q)$ with the Sylow $2$-subgroups of $\PSL(2,q)$.

\begin{lem}
\label{norm}
Let $q>13$ be an odd prime power, let $P\in\Syl_2(\SL(2,q))$ and let $T\in \Syl_2(\PSL(2,q))$. Then $N_{\SL(2,q)}(P)=P$ and  $N_{\PSL(2,q)}(T)=T$ when $q\not\equiv\pm3\pmod{8}$. If $q\equiv\pm3\pmod{8}$ then $|N_{\SL(2,q)}(P)|=3|P|$ and 
$|N_{\PSL(2,q)}(T)|=3|T|$. In particular, in both cases, $\nu_2(\SL(2,q))=\nu_2(\PSL(2,q))$.
\end{lem}

\begin{proof}
This can be deduced from  Dickson's classification of subgroups of $\SL(2,q)$ and $\PSL(2,q)$.
\end{proof}

As a consequence, we have the following.

\begin{cor}\label{selfPSL}
Let $q>13$ be an odd prime power. The map $\varphi:\Syl_2(\SL(2,q))\rightarrow\Syl_2(\PSL(2,q))$  defined by $\varphi(P)=\overline{P}$ is a bijection.
\begin{proof}
 We know that the map $\varphi$ is surjective and since $|\Syl_2(\SL(2,q))|=|\Syl_2(\SL(2,q))|$, we have that it is a bijection.
\end{proof}
\end{cor}

As a consequence, we have that an element $x \in (\SL(2,q))_2$ lies in a unique Sylow $2$-subgroup of $\SL(2,q)$ if and only if $\overline{x}$ lies in a unique Sylow $2$-subgroup of $\PSL(2,q)$. Therefore, it suffices to prove Theorem D in the case $\SL(2,q)$.

\begin{thm}
Let $q=2^k\pm 1$ be a prime power for an integer $k$. Then $G=\SL(2,q)$ does not possess redundant Sylow $2$-subgroups.
\begin{proof}
If $q\in \{3,5,7,9\}$, then the result can be checked by GAP \cite{gap}. Thus, we may assume that $q=2^k\pm 1$ with $k\geq 4$. Thus, if $P \in \Syl_2(G)$, then by Lemma \ref{norm}, we have that $P=N_G(P)$.  By Dickson's classification of subgroups of $\SL(2,q)$, we know that $P$ is  isomorphic to the generalized quaternion group of order $2^{k+1}$. Let $x \in P$ be an element of order $2^{k}$. Looking at the character table and the description of the elements of $\SL(2,q)$  given in  \cite[Chapter 38]{Dornhoff}, we deduce that if $y\in \SL(2,q)$ has order $2^{k}$, then  $y$  is conjugate to $x^t$ for some odd $t$ with $1\leq t \leq 2^{k-1}-1$. Therefore, we have  that there exist $2^{k-2}$ different conjugacy classes whose elements have order $2^{k}$ and that $$|x^G\cap P|=|y^G\cap P|$$ for every element $y$ of order $2^k$.  Now, since $P$ possesses $2^{k-1}$ elements of order $2^{k}$ we deduce that $|x^G\cap P|=2$ (in fact, $x^G\cap P=\{x,x^{-1}\}=x^P$). Finally, it is possible to prove that $C_G(x)=\langle x \rangle$ and hence $$|x^G|=2\frac{|G|}{|P|}=2|\Syl_2(G)|=|x^G\cap P||\Syl_2(G)|.$$ Thus, by Lemma \ref{necessary}, $x$ lies in a unique Sylow $2$-subgroup of $\SL(2,q)$ and the result follows.
\end{proof}
\end{thm}

Thus, it only remains to prove that if $q$ does not have the form $2^k\pm 1$, then $\SL(2,q)$ possesses redundant Sylow $2$-subgroups.

\begin{thm}\label{centra}
Let $q$ be a power of a prime such that $\{q-1,q,q+1\} \cap \{2^k|k \in \mathbb{N}\}=\varnothing$ and let $G=\SL(2,q)$. If  $x$ is a $2$-element of  $G$, then either $q-1$ or $q+1$ divides $|C_{G}(x)|$. 
In particular, 
$C_G(x)$
is not a $2$-group.
\end{thm}

\begin{proof}
Let $x$ be a $2$-element of $G$,  let $C=\langle x \rangle$, let $\F_q$ be the field of $q$ elements and let $V$ be the vector space of dimension $2$ over $\F_q$. We have that $C$ acts naturally on $V$ and since  the characteristic of the field defining $V$ is odd, $V$ is a completely reducible $C$-module by Maschke's Theorem (see \cite[Theorem 1.9]{Isaacscar}). Thus, either $C$ may be diagonalized over $V$ or  $C$ acts irreducibly on $V$.

Assume first that  $C$ can be diagonalized over $V$.  Therefore, we can identify $x$ with a matrix $\diag(e,e^{-1})$ for some $e\in \F_{q}^{\times}$.  We observe that $\diag(a,a^{-1})$ centralizes $x$ for every $a\in \F_q$. Now, since $q-1$ is not a power of $2$, we have that there exists $a\in \F_q$ of odd order and hence $\diag(a,a^{-1})$ is an element of odd order centralizing $x$.

Assume now that $C$ acts irreducibly on $V$. We claim  that $C_{G}(x)$ contains a cyclic group of order $q+1$. Let $K$ be the centralizer of $C$ in $\mathrm{End}(V)$. It follows from Schur's 
lemma (see \cite[Lemma 1.5]{Isaacscar}) that $K$ is a division ring. Wedderburn's 
theorem says that a finite division ring is a field. The multiplicative 
group $K^{\times}$ of $K$ is cyclic of order $k$. We know that $\F_q$ is contained in $K$. Since $C$ acts 
irreducibly on $V$ (and is contained in $K$), $K$ must be a proper field 
extension of $\F_q$. So $k \geq q^2$. Since $K^{\times}$ is cyclic and it acts faithfully 
on $V$, its order $k-1$ must be at most $|V|-1 = q^{2}-1$. It follows that $K^{\times}$ 
is a cyclic group of order $|V|-1$ moving the non-zero elements of V in 
one cycle. Such a group is called a Singer cycle. This is the 
centralizer of $C$ in $\GL(V)$. Take $T = K^{\times} \cap \SL(V)$. This group centralizes 
$C$. The order of $T$ is divisible by $(q^{2}-1)/(q-1) = q+1$ and the claim follows. Since  $q+1$ is not a power of $2$, we deduce that there exists an element of odd order centralizing $x$.
\end{proof}

The following result completes the proof of Theorem D. 

\begin{thm}
Let $q$ be a power of a prime such that $\{q-1,q,q+1\} \cap \{2^k|k \in \mathbb{N}\}=\varnothing$. Then  $\SL(2,q)$ possesses redundant Sylow $2$-subgroups.
\end{thm}

\begin{proof}
Let $G=\SL(2,q)$ and let $P \in \Syl_2(G)$. Since the cases  $q\leq 27$ can be checked by GAP \cite{gap}, we may assume that $q>27$.

Suppose first that $q\not\equiv\pm3\pmod{8}$. By Lemma \ref{norm}, we know that $P=N_G(P)$. By Lemma \ref{necessary}, it suffices to prove that  $|x^G|< \nu_p(G)|x^P|$ for all $x \in P$. This is equivalent to prove that $|C_G(x)|>|C_P(x)|$ for all $x\in P$. Thus, it suffices to prove that $C_G(x)$ is not a $2$-group for all $x \in P$. Thus, the result follows by Theorem \ref{centra}.
 
 Suppose now that $q\equiv\pm3\pmod{8}$. Then by Lemma \ref{norm} we know that $|N_G(P)|=3|P|$. Thus proving $|x^G|< \nu_p(G)|x^P|$ for all $x \in P$ is equivalent to proving that $3|C_P(x)|<|C_G(x)|$. Since  $q\equiv\pm3\pmod{8}$, we deduce that $|C_P(x)|\leq|P|\leq 8$. Thus, since $q>27$, we have that  $$3|C_P(x)|\leq 24<q-1\leq|C_G(x)|,$$
 where the last inequality follows fom Theorem \ref{centra}.  Thus we have proved that $|x^G|< \nu_p(G)|x^P|$ for all $x \in P$ and the result follows from Lemma \ref{necessary}.
\end{proof}

\section{Sporadic groups}

In this section, we collect the results on the sporadic groups which we have obtained using GAP \cite{gap}. For the smallest sporadic groups we used Lemma \ref{necessary} to study whether there is a redundant Sylow $p$-subgroup.

\begin{verbatim}
-Fully Checked Groups: ["M11","M12", "M22", "M23", "M24", "J1", "J2", 
"J3", "HS", "McL", "He", "Co3", "Ru", "Suz"];

-Groups and primes for which there are redundant Sylow p-subgroups:
[ [ "M12", 3 ], [ "M22", 3 ], [ "M23", 3 ], [ "M24", 3 ], [ "J1", 2 ],
[ "J2", 3 ], [ "HS", 3 ], [ "He", 3 ], [ "He", 5 ], [ "Co3", 2 ],
[ "Ru", 3 ], [ "Ru", 5 ], [ "Suz", 5 ] ].
\end{verbatim}

It does not seem possible with our hardware to decide the existence of a redundant Sylow $p$-subgroup for the larger sporadic groups using Lemma \ref{necessary}.
For these groups, we have used the sufficient condition for the existence of a redundant Sylow $p$-subgroup provided by Lemma \ref{nav}.
Using the character table and the second orthogonality relation, we can compute $|C_G(x)|$ for any sporadic group and any $p$-element $x\in G$. If for any $p$-element $x\in G$, $|C_G(x)|$ does not divide $|N_G(P)|$, where $P$ is a Sylow $p$-subgroup of $G$, then it follows from Lemma \ref{nav} that $G$ has a redundant Sylow $p$-subgroup. The structure of the Sylow normalizers of the sporadic groups can be found in \cite{wil}. Using this criterion, we have checked the following.

\begin{verbatim}
-Groups studied with this criterion: [ "J4", "Co1", "Co2",  "Fi22", 
"Fi23", "Fi24","ON", "HN", "Ly", "Th" "B", "M" ];

-Groups and primes for which there are redundant Sylow p-subgroups:
[ ["J4",3], ["Co1", 5], ["Fi23", 5], ["Fi24'", 5], ["Fi24'", 7],
 ["Ly'", 2], ["Th'", 2], ["B", 7], ["M", 5], ["M", 7], ["M", 11] ].
\end{verbatim}

Since Lemma \ref{nav} is not a sufficient  condition for the existence of redundant Sylow subgroups, there may exist some other pairs $(G, p)$ with $G$ sporadic and $p$ prime such that $G$ has a  redundant Sylow $p$-subgroup.

It is worth remarking that the Monster has a redundant Sylow $7$-subgroup. This is related to the work in \cite{hv1, hv2}.

  \section{On the number of Sylow $p$-subgroups}
  
We have conjectured in \cite{mmm} that if $G$ is a finite group generated by its $p$-elements, then $G_p$ cannot be covered by $p$ proper subgroups. Our next goal is to prove that if $G$ does not have a normal Sylow $p$-subgroup, then $G_p$ cannot be covered by $p$ Sylow $p$-subgroups. The following lemma is proved in \cite[p. 79-80]{mbd}. For the convenience of the reader, we present a proof in modern language and notation.

  \begin{lem}
  \label{pel}
  Let $G$ be a group with non-normal Sylow $p$-subgroups of order $p^n$. Then $|G_p|\geq p^{n+1}$.
  \end{lem}
  
  \begin{proof}
  Let $P,Q\in\Syl_p(G)$ such that $P\neq Q$ and  $P\cap Q$ has order as large as possible. Write $D=P\cap Q$.  Since $D$ is a $p$-group, there exists $g\in N_P(D)-D$. In particular, $g\not\in Q$. Since $g$ is a $p$-element, $g\not\in N_G(Q)$, and we deduce that $Q,Q^g,\dots,Q^{g^{p-1}}$ are $p$ pairwise different Sylow $p$-subgroups of $G$.  Note that 
  $$
  D=P\cap Q=(P\cap Q)^g=P\cap Q^g,
  $$ 
  so $D=P\cap Q^{g^i}$ for any $i=0,\dots,p-1$. 
  On the other hand, since $g$ normalizes $D$, $D$ is contained in $Q^{g^i}\cap Q^{g^j}$ for any $0\leq i< j\leq p-1$. By the maximality of $D$, we conclude that $D=Q^{g^i}\cap Q^{g^j}$ for any $0\leq i< j\leq p-1$. Thus, we have found $p+1$ Sylow $p$-subgroups of $G$ such that the intersection of any two of them is $D$. Put $|D|= p^{m}$. It follows that
  $$
  |G_p|\geq|P\cup\bigcup_{i=0}^{p-1}Q^{g^i}|\geq(p+1)(p^n-p^m)+p^m=p^{n+1}+p^n-p^{m+1} \geq p^{n+1}.
  $$
  \end{proof}
  
  \begin{thm}
  Let $G$ be a group with non-normal Sylow $p$-subgroups. Then $G_p$ cannot be covered by $p$ Sylow $p$-subgroups. In particular, if $\nu_p(G)=p+1$, then $G$ does not have a redundant Sylow $p$-subgroup.
  \end{thm}
  
  \begin{proof}
  Write $|G|_p=p^n$. We have $|G_p|\geq p^{n+1}$ by Lemma \ref{pel}. On the other hand, the union of $p$ Sylow $p$-subgroups has less than $p^{n+1}$ elements since the identity belongs to every Sylow subgroup. The result follows.
  \end{proof}
  
  We can improve the second part of this result. We need the following, which was essentially proved by Richard Lyons in \cite{lyo}. 
    
 \begin{lem}
 \label{p2}
 If $G$ is a group with at most $p^2$ Sylow $p$-subgroups, then $|G:O_p(G)|_{p} = p$.
 \end{lem}
 
 \begin{proof}
 Write $|G|_p=p^n$. Let $P$, $Q \in \Syl_p(G)$ be two different Sylow $p$-subgroups. We claim that $|P\cap Q|=p^{n-1}$. Consider the conjugation action of $P$ on the set $\Syl_p(G) \setminus \{ P \}$ of cardinality less than $p^2$. Each orbit must have size $p$. In particular, $|N_P(Q)|=p^{n-1}$. Since $N_P(Q)Q$ is a $p$-group, it follows that $N_P(Q)\leq Q\cap P$, whence $|P\cap Q|=p^{n-1}$.
 
 Let $P \in \Syl_{p}(G)$. The Frattini subgroup $\Phi(P)$ of $P$ is contained in every Sylow $p$-subgroup of $G$ by the previous paragraph. Since the intersection of all Sylow $p$-subgroups of $G$ is $O_{p}(G)$, we find that $\Phi(P) \leq O_{p}(G)$. 
 
 The group $\overline{G}=G/O_p(G)$ has elementary abelian Sylow $p$-subgroups. By Brodkey's theorem \cite[Theorem 1.37]{isa}, there exist $P$, $Q\in\Syl_p(G)$ such that $\overline{P}\cap\overline{Q}=\overline{1}$, so $P\cap Q=O_p(G)$. It follows from the first paragraph that $|P:O_{p}(G)| = p$. 
 \end{proof} 
 
 \begin{thm}
 Let $G$ be a finite group with at most $p^2$ Sylow $p$-subgroups. Then $G$ does not have a redundant Sylow $p$-subgroup.
 \end{thm}
 
 \begin{proof}
 This follows from Lemma \ref{p2} and Corollary \ref{op}.
 \end{proof}
 
 We think that it should be possible to weaken the hypothesis that $G$ has at most $p^2$ Sylow $p$-subgroups in the previous theorem. 
 This result suggests the following question.
  
  \begin{que}
  What is the smallest value of $n=n(p)$ such that there exists a group $G$ with a redundant Sylow $p$-subgroup and $\nu_p(G)=n$? 
  \end{que}
  
  It also seems interesting to study the minimal size of a covering set by Sylow $p$-subgroups of a group with a redundant Sylow $p$-subgroup. One of the  smallest such examples, {\tt SmallGroup}(108,17), has $27$ Sylow $2$-subgroups and we have checked with GAP \cite{gap} that the smallest size of a covering of $G_p$ by Sylow $p$-subgroups is at most $12$. It would be interesting to see if when $G$ is a group with a redundant Sylow $p$-subgroup it is necessarily the case that there are ``many" redundant Sylow $p$-subgroups. For instance, we have the following question.
  
  \begin{que}
  Is it true that there exists $c<1$ (possibly depending on $p$) such that if $p$ is a prime and $G$ is a finite group with a redundant Sylow $p$-subgroup, then $G_p$ can be covered by at most $c\nu_p(G)$ Sylow $p$-subgroups?
  \end{que}
  
  In other words, this question is asking whether in case that there is a redundant Sylow $p$-subgroup, there necessarily exists a positive proportion of redundant Sylow $p$-subgroups.
  The following dual question seems also of interest.
  
  \begin{que}
  Is it true that there exists $c>0$ (possibly depending on $p$) such that if $p$ is a prime and $G$ is a finite group, then $G_p$ cannot be covered by less than $c\nu_p(G)$ Sylow $p$-subgroups?
  \end{que}
  
  Note that Conjecture A of \cite{ghe} implies that $G_p$ cannot be covered by less than $\nu_p(G)^{1-\frac{1}{p}}$ Sylow $p$-subgroups.  These questions are partially inspired by a theorem of Sambale and T\u{a}rn\u{a}uceanu \cite{st}. 
  They  proved in \cite{st} that there exists $c=c(n)>0$ such that  if a finite group $G$ is not covered by $H_1,\dots, H_n$ then the proportion of elements of $G$ in $G\setminus(H_1\cup\dots\cup H_n)$ is at least $c$. (Actually, they proved stronger and more precise results.) We conclude with a  possible $p$-version of this.

\begin{que}
Does there exist $c=c(n)>0$ (possibly depending on $p$ too) such that if $G$ is a finite group and $G_p$ is not covered by $P_1,\dots, P_n\in\Syl_p(G)$ then $$|G_p\setminus(P_1\cup\dots\cup P_n)|/|G_p|\geq c?$$
\end{que}

{\bf Acknowledgement:} We thank Gunter Malle for pointing out gaps in the proofs of Theorem D and  Theorem 5.1 of the published version.


\begin{thebibliography}{99}

\bibitem{ag}
J. Almeida, M. Garonzi, 
On minimal coverings and pairwise generation of some primitive groups of wreath product type, 
2023, arXiv:2301.03691.

%\bibitem{bg}
%T.\,C. Burness, R.\,M. Guralnick, On the generation of simple groups by Sylow subgroups, arXiv:2204.04311v2.

\bibitem{Blau} H. I. Blau, On trivial intersection of cyclic Sylow subgroups. \emph{Proc. Amer. Math. Soc.} \textbf{94} (1985), no. 4, 572--576.

\bibitem{BEGHM} J. R. Britnell, A. Evseev, R. M. Guralnick, P. E. Holmes, A. Mar\'oti, Sets of elements that pairwise generate a linear group. \emph{J. Combin. Theory Ser. A} \textbf{115} (2008), no. 3, 442--465.

\bibitem{BFS} R. A. Bryce, V. Fedri, L. Serena, Subgroup coverings of some linear groups. \emph{Bull. Austral. Math. Soc.} \textbf{60} (1999), no. 2, 227--238.

\bibitem{bgmn}
T.\,C. Burness, R.\,M. Guralnick, A. Moret\'{o}, and G. Navarro,
\rm{On the commuting probability of p-elements in a finite group}. \emph{Algebra \& Number Theory} \textbf{17} (2023), no. 6, 1209--1229.

%\bibitem{cam}
%P. Cameron, Graphs defined on groups. \emph{Int. J. Group Theory} {\bf 11} (2022), 53--107.

\bibitem{car}
R. Carter,  \textit{Finite Groups of Lie Type. Conjugacy Classes and Complex Characters}. Wiley Classics LIbrary, Chichester, 1993.

\bibitem{cf}
R. Carter and P. Fong, The Sylow 2-subgroups of the finite classical groups. \emph{J. Algebra} {\bf 1 (1)} (1964), 139--151.

\bibitem{coh}
J. H. E. Cohn, On $n$-sum groups. 
\emph{Math. Scand.} \textbf{75} (1994), no. 1, 44--58.	

\bibitem{Dickson} L. E.  Dickson, \textit{Linear Groups with an Exposition of the Galois Field Theory}. Dover Publications, New York, 1958.

\bibitem{Dornhoff} L. Dornhoff, \textit{Group Representation Theory. Part A: Ordinary representation theory}. Marcel Dekker, Inc., New York, 1971. Pure and Applied Mathematics, 7.

\bibitem{f}
W. Findlay, The Sylow subgroups of the symmetric group. \emph{Trans. Amer. Math. Soc.} {\bf 24} (1904), 263--278.

\bibitem{FGM} F. Fumagalli, M. Garonzi, A. Mar\'oti, On the maximal number of elements pairwise generating the symmetric group of even degree. \emph{Discrete Math.} \textbf{345} (2022), no. 4, Paper No. 112776.


%\bibitem{fg}
%F. Fumagalli, M. Garonzi, On the primary coverings of finite solvable and symmetric groups. \emph{J. Group Theory} {\bf 24} (2021), 1189--1121.

\bibitem{gap}
The GAP Group, \rm{GAP Groups, algorithms, and programming, version 4.12.1, 2022}, {http://www.gap-system.org}.

\bibitem{ghe}
P. Gheri, On the number of $p$-elements in a finite group. 
\emph{Ann. Mat. Pura Appl.} \textbf{200} (2021), 1231--1243. 

\bibitem{her}
M. Herzog, On $2$-Sylow intersections. 
\emph{Israel J. Math. } \textbf{11} (1972), 326--327.

\bibitem{hv1}
C. Y. Ho, H. V\"olklein, 
A criterion for an element to belong to a given Sylow p-subgroup. I., \emph{J. Algebra} \textbf{132} (1990), 113--122.

\bibitem{hv2}
C. Y. Ho, H. V\"olklein, 
A criterion for an element to belong to a given Sylow p-subgroup. II., \emph{Geom. Dedicata} \textbf{28} (1988), 363--368.

\bibitem{HM} P. E. Holmes, A. Mar\'oti, Pairwise generating and covering sporadic simple groups. \emph{J. Algebra} \textbf{324} (2010), no. 1, 25--35. 

%\bibitem{hal}
%M. Hall, On the number of Sylow subgroups in a finite group. \emph{J. Algebra} {\bf 7} (1967), 363--371.

%\bibitem{HMM} N. N. Hung, J. Martínez, A. Mar\'oti, Conjugacy classes of pi-elements and %nilpotent/abelian Hall pi-subgroups. To appear in Pacific Journal of Mathematics.

%\bibitem{Huppert} B. Huppert, \textit{Endliche Gruppen I}. Springer, Berlin, 1967.

%\bibitem{FG} F. Fumagalli, M. Garonzi, On the primary coverings of finite solvable and symmetric groups. \emph{J. Group Theory} \textbf{24} (2021), no. 6, 1189--1211.


%\bibitem{Gustafson}
%W.\,H. Gustafson, \rm{What is the probability that two elements
%	commute?}, \textit{Amer. Math. Monthly} \textbf{80} (1973),
%1031--1034.

\bibitem{Ho} C. Y. Ho, Finite groups in which two different Sylow $p$-subgroups have trivial intersection for an odd prime $p$. \emph{J. Math. Soc. Japan} \textbf{31} (1979), no. 4, 669--675.

\bibitem{Isaacscar} I. M. Isaacs, \textit{Character Theory of Finite Groups}. Dover Publications, New York, 1976.

\bibitem{isa}
I. M. Isaacs, \textit{Finite group theory}, Amer. Math. Soc.
Providence, Rhode Island, 2008.


%\bibitem{Jacobson} N. Jacobson, \textit{Basic Algebra I}. Freeman, San Francisco, 1974.

%\bibitem{Jones} G. A. Jones, Cyclic regular subgroups of primitive permutation groups.
%\emph{J. Group Theory} \textbf{5} (2002), no. 4, 403--407.

\bibitem{Garonzi} M. Garonzi, Finite groups that are the union of at most $25$ proper subgroups. \emph{J. Algebra Appl.} \textbf{12} (2013), no. 4, 1350002.

\bibitem{GKS} M. Garonzi, L-C. Kappe, E. Swartz, On integers that are covering numbers of groups. \emph{Exp. Math.} \textbf{31} (2022), no. 2, 425--443. 

\bibitem{GarMar} M. Garonzi, A. Mar\'oti, Covering certain wreath products with proper subgroups. \emph{J. Group Theory} \textbf{14} (2011), no. 1, 103--125.

\bibitem{kon}
A. S. Kondrat'ev, Normalizers of the Sylow $2$-subgroups in finite simple groups. \emph{Math. Notes} {\bf 78} (2005), 338--346. 



\bibitem{Lucido} M. S. Lucido, On the covers of finite groups. Groups St. Andrews 2001 in Oxford. Vol. II, 395--399,
London Math. Soc. Lecture Note Ser., 305, Cambridge Univ. Press, Cambridge, 2003.

\bibitem{lyo} 
R. Lyons,
https://mathoverflow.net/questions/373906/group-with-fewer-than-p2-sylow-p-subgroups.


%\bibitem{Ma2} A. Mar\'oti, On the orders of primitive groups. \emph{J. Algebra} \textbf{258} (2002), no. 2, 631--640.	

\bibitem{mar}
A. Mar\'oti, Covering the symmetric groups with proper subgroups. \emph{J. Combin. Theory Ser A.} {\bf 110} (2005), 97--111.

\bibitem{mmm}
A. Mar\'oti, J. Mart\'{\i}nez, A. Moret\'o, Covering the set of $p$-elements in a finite group by proper subgroups, submitted.

\bibitem{mbd}
G. A. Miller, H. F. Blichfeldt, L. E. Dickson, Theory and applications of finite groups, Dover, New York, 1961.

\bibitem{navb}
G. Navarro, Characters and Blocks of Finite Groups, Cambridge University Press, 1998. 

%\bibitem{nav}
%G. Navarro, Number of Sylow subgroups in $p$-solvable groups. \emph{Proc. Amer. Math. Soc.} {\bf 131} (2003), 3019--3020. 

\bibitem{Neumann} B. H. Neumann, Groups covered by finitely many cosets. \emph{Publ. Math. Debrecen} \textbf{3} (1954), 227--242.

%\bibitem{Pyber} L. Pyber, The number of pairwise noncommuting elements and the index of the centre in a finite group. \emph{J. London Math. Soc.} (2) \textbf{35} (1987), no. 2, 287--295.

\bibitem{sam}
B. Sambale, Blocks of Finite Groups and their Invariants, Springer, 2014. 


\bibitem{st}
B. Sambale, M. T\u{a}rn\u{a}uceanu, On the size of coset unions. 
\emph{J. Algebraic  Combin.} \textbf{55} (2022), no. 3, 979--987.	

\bibitem{sch}
J. Schmidt, https://math.stackexchange.com/questions/299682/if-p-is-an-odd-prime-does-every-sylow-p-subgroup-contain-an-element-not-in?noredirect=1 lq=1

\bibitem{suz}
M. Suzuki, Finite groups of even order in which Sylow 2-groups are independent. \emph{Ann. of Math.} {\bf 80} (1964), 58--77.

\bibitem{Sw} E. Swartz, On the covering number of symmetric groups having degree divisible by six. \bibitem{Discrete Math.} \textbf{339} (2016), no. 11, 2593--2604. 

\bibitem{tom}
M. J. Tomkinson, Groups as the union of proper subgroups. 
\emph{Math. Scand.} \textbf{81} (1997), no. 2, 191--198.	

\bibitem{tur}
A. Turull, Generic fixed point free action of arbitrary groups. \emph{Math. Z.} {\bf 187} (1984), 491--503.

\bibitem{wil}
R. A.  Wilson, The McKay conjecture is true for the sporadic simple groups. 
\emph{J. Algebra} {\bf 207} (1998), 294--305.


\end{thebibliography}
\end{document}